\documentclass[10pt]{amsart}
\usepackage{graphicx}
\usepackage{amssymb,amsmath}
\def\dint{\displaystyle\int}
\def\dsum{\displaystyle\sum}
\def\dprod{\displaystyle\prod}

\def\ub{\textbf u}

\newtheorem{thm}{Theorem}
\newtheorem{thms}{Theorem}[section]
\newtheorem{lemma}{Lemma}[section]
\newtheorem{Def}{Definition}[section]
\newtheorem{prop}{Proposition}[section]
\newtheorem{cor}{Corrolary}[section]

\begin{document}

\title{Equivariant mirror symmetry for the weighted projective line}
\date{\today}
\author{Dun Tang}
\maketitle
\begin{abstract}
In this paper, we establish equivariant mirror symmetry for the weighted projective line. This extends the results by B. Fang, C.C. Liu and Z. Zong, where the projective line was considered [{\it Geometry \& Topology} 24:2049-2092, 2017].  More precisely, we prove the equivalence of the $R$-matrices for A-model and B-model at large radius limit, and establish isomorphism for $R$-matrices for general radius. We further demonstrate that the graph sum of higher genus cases are the same for both models, hence establish equivariant mirror symmetry for the weighted projective line.
\end{abstract}

\tableofcontents

\section{Introduction}

In mirror symmetry for toric varieties, one aims at establishing equivalence between A-model invariants and the Ginzburg-Landau B-model invariants for a given toric variety.

On the A-model side, there are extensive studies for the equivariant Gromov-Witten theory. In \cite{Given1}, A.B. Givental computed all genus descendent invariants of equivariant Gromov-Witten theory of tori action with isolated fixed points for the smooth case. The process of recovering higher genus data is now known as Givental's formula. In \cite{Zong}, Z. Zong gave all genus equivariant Gromov-Witten invariants for GKM orbifolds by generalizing Givental's formula to the orbifold case.

On the B-model side, B. Eynard and N. Orantin discovered a way to compute the topological expansion of matrix integrals in \cite{Eynard}. Eynard-Orantin's topological recursion is related to Givental's formula \cite{Dunin}.

More recently, B. Fang, C.C. Liu and Z. Zong established the equivariant mirror symmetry for the projective line. They directly computed the $R$-matrices of both A and B-models, and applied Givental's formula and Eynard-Orantin's recursion, thereby proved the equivariant mirror symmetry for the projective line by calculating graph sums  \cite{Fang1}.

In this paper, we extend the equivariant mirror symmetry to the weighted projective line.
First, we associate the equivariant Gromov-Witten invariants of the weighted projective line to the Eynard-Orantin invariants of the affine curve determined by the superpotential of its $T$-equivariant Landau-Ginzburg mirror. It is proved by calculating the graph sum and applying the main results in \cite{Eynard}\cite{Given}\cite{Given1}. We use the equivalence of $R$-matrices in both models, which is established by computations with quantum Riemann-Roch, and integration on Lefschetz thimble at the large radius limit \cite{Tseng10}.

Secondly, we establish a precise correspondence between A-model genus $g, N$ point descendent equivariant Gromov-Witten invariants, and the Laplace transform of B-model Eynard-Orantin invariant along Lefschetz thimbles.

\section*{Acknowledgment}
The author would like to thank Professor Bohan Fang for leading me into this field, and his stimulating advices throughout the study and preparation of the paper.	The research project is financially supported by the National Innovation Program for Undergraduate Research.

\section{The Frobenius Manifold}
We assume $m$ and $n$ are coprime throughout this paper, i.e., $W\mathbb{P}(m,n)$ has trivial generic stabilizer.
We use bold form quantities $\bold{w_\ell,w_{-\ell},p,q_\ell,q_0,q_{-\ell}}$ to represent cohomology classes and normal form quantities $w_1,w_2,p,q_l,q_0,q_{-l}$  to represent complex numbers.

A weighted projective line $W\mathbb{P}(m,n)$ is given by the fan below by standard toric construction:
\begin{figure}[h]
\begin{center}
\setlength{\unitlength}{2mm}
\begin{picture}(20,5)
\put(12,4){\vector(-1,0){10}}
\put(12,4){\vector(1,0){6}}
\put(12,3){${}^|$}
\put(1,1){$-m$}
\put(17,1){$n$}
\end{picture}
  \label{fig0}
\end{center}
\end{figure}
\vspace{-5mm}

\noindent i.e., $W\mathbb{P}(m,n)=(\mathbb{C}^2\setminus\{(0,0)\})/\sim$, where $\sim$ is defined by $(z_1,z_2)\sim (\lambda^n z_1,\lambda^m z_2)$. A point in $W\mathbb{P}(m,n)$ is given by homogeneous coordinates $[z_1:z_2]$. Let $T=(\mathbb{C}^\ast)^2$ act on $W\mathbb{P}(m,n)$ by $(t_1,t_2)\cdot[z_1,z_2]=[t_1z_1,t_2z_2]$.

\subsection{Jacobian ring}
\

For $y\in\mathbb{C}$, let $Y=e^y\in \mathbb{C}^\ast$.
Define the equivariant superpotential $W_T:\mathbb{C}^\ast\to \mathbb{C}$ by
\[
W_T(Y)= Y^m+ \dsum_{\ell=1}^{m-1}\tilde{q}_\ell Y^\ell + \dsum_{\ell=1}^{n}\tilde{q}_{-\ell} Y^{-\ell} + \tilde{w}_m\log (Y^m) +\dsum_{\ell=1}^{m-1}\tilde{w}_\ell\log(\tilde{q}_\ell Y^\ell) + \dsum_{\ell=1}^{n} \tilde{w}_{-\ell}\log(\tilde{q}_{-\ell}  Y^{-\ell}).
\]

Let $x=W_T(e^y)$. The Jacobian ring of $W_T$ is
\[
\mathrm{Jac}(W_T)= \mathbb{C}[w][Y]/\left<\dfrac{\partial W_T}{\partial y}\right> = \mathbb{C}[w][Y]/\left<mY^m +  \dsum_{\ell=1}^{m-1}\ell \tilde{q}_\ell Y^\ell - \dsum_{\ell=1}^{n} \ell \tilde{q}_{-\ell} Y^{-\ell} +\tilde{p}\right>,
\]
where $\tilde{p}=\sum_{\ell=1}^m \ell \tilde{w}_\ell - \sum_{\ell=1}^n \ell \tilde{w}_{-\ell}$.

Let $\{p^\alpha\}$ be the set of critical points of $W_T$. Define residual pairing $(f,g)$ on $\mathrm{Jac}(W_T)$ by $(f,g)= \dsum_{\alpha} \mathrm{Res}_{p^\alpha} \dfrac{f(Y)g(Y)}{\big({\partial W_T}/{\partial y}\big)} \dfrac{dY}Y$.


\subsection{Equivariant Gromov-Witten potentials}
\

The Chen-Ruan cohomology for $W\mathbb{P}(m,n)$ is
$H_{CR}^\ast(W\mathbb{P}(m,n))=\mathbb{C}[X,\bar{X}]/\left<X\bar{X},mX^m-n\bar{X}^n\right>$
\cite{Adem07}. For more details, please refer to \cite{Adem07} and the original papers \cite{Chen04,Chen01,Chen}.

Let $\mathbb{C}[\bold{w}]=\mathbb{C}[\bold{w_{-n},\cdots,w_{-1},w_1,\cdots,w_m}]$ be the equivariant cohomology ring of a point with $T=\mathbb{T}^{m+n}$ action. Example 105 in \cite{Liu13} gives the equivariant Chen-Ruan cohomology $H_T^\ast(W\mathbb{P}(m,n))\cong \mathbb{C}[\bold{w}][X,\bar{X},\bold{p}]/\left<X\bar{X},mX^m-n\bar{X}^n+\bold{p}\right>$,
with $\bold{p}=\sum_{\ell=1}^m \ell \bold{w}_\ell - \sum_{\ell=1}^n \ell \bold{w}_{-\ell}$.

Suppose that $d>0$ or $2g-2+N>0$, then $\bar{\mathcal{M}}_{g,n}(W\mathbb{P}(m,n),d)$ is well defined. Let $\mathbb{L}_i$ be the i-th tautological bundle, whose restriction to a point $[\Sigma,x_1,\cdots,x_n]\in \bar{\mathcal{M}}_{g,n}(W\mathbb{P}(m,n),d)$ is isomorphic to $T_{x_i}^\ast \Sigma$.
Let $\psi_i=c_1(\mathbb{L}_i)$.
Let $\mathrm{ev}_i: \bar{\mathcal{M}}_{g,n}(W\mathbb{P}(m,n),d) \to W\mathbb{P}(m,n)$ be the evaluation at the $i$-th point.
For $\gamma_1,\cdots,\gamma_N\in H_T^\ast(W\mathbb{P}(m,n),\mathbb{C}), a_1,\cdots, a_N\in \mathbb{Z}_{\geq 0}$, define the orbifold descendent Gromov-Witten invariants as
\[\left< \tau_{a_1}(\gamma_1)\cdots \tau_{a_N}(\gamma_N) \right>_{g,N,d}^{W\mathbb{P}(m,n),T} = \dint_{[\bar{\mathcal{M}}_{g,N}(W\mathbb{P}(m,n),d)]^{\mathrm{vir}}} \dprod_{j=1}^N \psi_j^{a_j} \mathrm{ev}_j^\ast(\gamma_j) \in \mathbb{C}[\bold{w}].
\]

For $2g-2+M+N>0$ and given $\gamma_1,\cdots,\gamma_{M+N}\in H_T^\ast(W\mathbb{P}(m,n))$, define
\[
\begin{split}
& \left<  \dfrac{\gamma_1}{z_1-\psi_1},\cdots,\dfrac{\gamma_N}{z_N-\psi_N},\gamma_{N+1},\cdots,\gamma_{N+M} \right>_{g,M+N,d}^{W\mathbb{P}(m,n),T} \\
=& \dsum_{a_1,\cdots,a_N\geq 0} \left< \tau_{a_1}(\gamma_1)\cdots\tau_{a_N}(\gamma_N)\tau_0(\gamma_{N+1})\cdots\tau_0(\gamma_{N+M}) \right>_{g,M+N,d}^{W\mathbb{P}(m,n),T} \dprod_{j=1}^N z_j^{-a_j-1}.
\end{split}
\]
This is actually the formal expansion of $\frac{\gamma_i}{z_i-\psi_i}$ at $z_i^{-1}=0$.
In the unstable case $g=0,d=0,M+N=1 \;\mathrm{ or } \;2$, we define
\[
\begin{split}
\left< \frac{\gamma_1}{z_1-\psi_1}\right>_{0,1,0}^{W\mathbb{P}(m,n),T} & = z_1\int_{W\mathbb{P}(m,n)} \gamma_1,\\
\left< \frac{\gamma_1}{z_1-\psi_1},\gamma_2\right>_{0,2,0}^{W\mathbb{P}(m,n),T} & = \int_{W\mathbb{P}(m,n)} \gamma_1\cup\gamma_2,\\
\left< \frac{\gamma_1}{z_1-\psi_1}, \dfrac{\gamma_2}{z_2-\psi_2}\right>_{0,2,0}^{W\mathbb{P}(m,n),T} & = \dfrac1{z_1+z_2}\int_{W\mathbb{P}(m,n)} \gamma_1\cup\gamma_2.
\end{split}
\]

Let $\mathbf{t}=\sum t^iT_i$, where $\{T_i\}$ form a basis of $H_T^\ast(W\mathbb{P}(m,n),\mathbb{C})$.
Suppose that $2g-2+N+M>0$ or $N>0$. Given $\gamma_1,\cdots,\gamma_{M+N}\in H_T^\ast (W\mathbb{P}(m,n))$, we define
\[
\begin{split}
& \left<\!\left< \dfrac{\gamma_1}{z_1-\psi_1}, \cdots,\dfrac{\gamma_N}{z_N-\psi_N},\gamma_{N+1},\cdots,\gamma_{N+M}\right>\!\right>_{g,N+M}^{W\mathbb{P}(m,n),T} \\  = &  \dsum_{d\geq 0}\dsum_{\ell\geq 0} \dfrac{Q^d}{\ell!}  \left< \dfrac{\gamma_1}{z_1-\psi_1}, \cdots,\dfrac{\gamma_N}{z_N-\psi_N},\gamma_{N+1},\cdots,\gamma_{N+M},\underbrace{\mathbf{t},\cdots,\mathbf{t}}_{\ell\;\mathrm{ times }}\right>_{g,M+N+\ell,d}^{W\mathbb{P}(m,n),T}.
\end{split}
\]

Let
\[
\begin{split}
\mathbf{u_j}&=\dsum_{a\geq 0} (u_j)_a z^a,\\
F_{g,N}^{W\mathbb{P}(m,n),T}(\mathbf{u_1\cdots u_N,t})&=\dsum_{a_1,\cdots,a_N\geq 0}\left<\!\left<\tau_{a_1}((u_1)_{a_1}),\cdots,\tau_{a_N}((u_N)_{a_N})\right>\!\right>_{g,n}^{W\mathbb{P}(m,n),T}.
\end{split}
\]

For fixed $M,N\in\mathbb{Z}_{\geq 0}$, consider $\pi:\bar{\mathcal{M}}_{g,N+M}(W\mathbb{P}(m,n),d)\to \bar{\mathcal{M}}_{g,N}$ which forgets the target and the last $M$ marked points and stabilizes it.
Let $\bar{\mathbb{L}}_i$ be the pull-back of $\mathbb{L}_i$ along $\pi$.
Let $\bar{\psi}_i=\pi^\ast(\psi_i)=c_1(\bar{\mathbb{L}}_i)$ be the pull-back of $\psi$-classes in $\bar{M}_{g,N}$. Define
\[
\bar{F}_{g,N}^{W\mathbb{P}(m,n),T}(\mathbf{u_1\cdots u_N,t}) =  \dsum_{M,d,a_1,\cdots,a_N\geq 0} \dfrac{Q^d}{M!N!}\cdot\dint_{[\bar{\mathcal{M}}_{g,M+N}(W\mathbb{P}(m,n),d)]^{\mathrm{vir}}} \dprod_{j=1}^N \mathrm{ev}_j^\ast((u_j)_{a_j})\bar{\psi_j}^{a_j} \dprod_{i=1}^M \mathrm{ev}_{i+N}^\ast(\mathbf{t}).
\]
Let $F_{g,N}^{W\mathbb{P}(m,n),T}(\mathbf{u,t}),\bar{F}_{g,N}^{W\mathbb{P}(m,n),T}(\mathbf{u,t})$ be such that all $\mathbf{u_j}=\mathbf{u}$.

We define the total descendent potential and the ancestor potential of $W\mathbb{P}(m,n)$ as follows, where $\hbar$ is an arbitrary parameter:
\[
\begin{split}
&D^{W\mathbb{P}(m,n),T}(\mathbf{u})=\exp\big( \sum_{N,g}\hbar^{g-1}F_{g,N}^{W\mathbb{P}(m,n),T}(\mathbf{u,0})\big),\\
&A^{W\mathbb{P}(m,n),T}(\mathbf{u,t})=\exp\big(\sum _{N,g}\hbar^{g-1}\bar{F}_{g,N}^{W\mathbb{P}(m,n),T}(\mathbf{u,t}) \big).
\end{split}
\]

In fact, by Givental's formula \cite{Given}, we have
\[D^{W\mathbb{P}(m,n),T}(\mathbf{u})=\exp(F_1^{W\mathbb{P}(m,n),T}) \hat{\mathcal{S}}^{-1} A^{W\mathbb{P}(m,n),T}(\mathbf{u,t}),\] where $F_1^{W\mathbb{P}(m,n),T}=\sum_N F_{1,N}^{W\mathbb{P}(m,n),T}$.
We shall give an equivalent graph sum formula in 3.2.

\subsection{Quantum cohomology}
\

Following Iritani \cite{Iritani}, we first compute the non-equivariant quantum cohomology ring $QH^\ast(W\mathbb{P}(m,n))$ with all twisted sectors added.

First we give a generalized definition of toric varieties. A toric variety is constructed from the following data:
\begin{itemize}
  \item an r-dimensional algebraic torus $\mathbb{T}\cong (\mathbb{C^\ast})^r$. Set $\mathbb{N}=\mathrm{Hom}(\mathbb{C}^\ast,\mathbb{T})$;
  \item $M$ elements $D_i\in \mathbb{M} =\mathrm{Hom}(\mathbb{T},\mathbb{C}^\ast)$, such that $\mathbb{M}\otimes\mathbb{R}=\dsum\mathbb{R}\cdot D_i$;
  \item a vector $\eta\in \mathbb{M}\otimes\mathbb{R}$ (that defines the stability condition).
\end{itemize}

The elements $\{D_i\}$ define a homomorphism $\mathbb{T}\to(\mathbb{C}^\ast)^M$. Let $\mathbb{T}$ act on $\mathbb{C}^M$ via this homomorphism.
Set $\mathcal{A}=\left\{I:\sum_{i\in I} \mathbb{R}_{>0} D_i\ni \eta \right\}, \mathcal{U}_\eta=\mathbb{C}^M\backslash \displaystyle\bigcup$$_{_{I\notin\mathcal{A}}}\mathbb{C}^I$.
A toric orbifold is defined by the quotient stack $\mathcal{X}=[\mathcal{U}_\eta/\mathbb{T}]$. Note that by setting $\eta=0$ we obtain the original definition for toric varieties.
Let $I_0=\displaystyle\bigcap$$_{_{I\in\mathcal{A}}}I, \quad \mathcal{A}'=\left\{ I-I_0|I\in\mathcal{A}\right\}$.
Let $\overline{D}_i$ be the image of $D_i$ in $\mathbb{M}/\sum_{i\in I_0}\mathbb{Z}D_i\cong H_{CR}^2(\mathcal{X},\mathbb{Z})$.


Choose an integral basis $\{e_a\}$ of $\mathbb{M}$. 
Assume that some elements in $\{e_a\}$ generates $\sum_{i\in I_0} \mathbb{R}_{>0} D_i$.
Let $\bar{e}_a$ be the image of $e_a$ in $H^2(\mathcal{X},\mathbb{R})$.
Let $m_{ia}$ be a matrix such that $D_i=\sum_{a} m_{ia}e_a$, where $m_{ia}\in\mathbb{Z}$. Note that $\overline{D}_i=\sum_a m_{ia}\bar{e}_a$, and that $\overline{D}_i=0$ if $i\in I_0$.
Let $\mathbb{K}_{e\!f\!f}=\left\{ d\in\mathbb{N}\otimes\mathbb{Q} : \{i:\left<D_i,d\right>\in\mathbb{Z}_{\geq 0} \}\in\mathcal{A} \right\}$.
Let $q_a$ be coordinates on $\mathrm{Hom}(\mathbb{N},\mathbb{C}^\ast)$ with respect to the dual basis of $\{e_a \}$.

\begin{Def}
The $I$-function of $\mathcal{X}$ is an $H_{CR}^\ast (\mathcal{X})$-valued power series defined by
\[
I(q,z)=e^{(\sum_\alpha \bar{e}_a \log q_a)/z} \sum_{d\in\mathbb{K}_{e\!f\!f}} q^d \cdot \dfrac{\prod_{\nu,i:\left<D_i,d\right>\leq \nu<0}(\overline{D}_i+(\left<D_i,d\right>-\nu)z)}{\prod_{\nu,i:\left<D_i,d\right>> \nu\geq 0}(\overline{D}_i+(\left<D_i,d\right>-\nu)z)}\mathbf{1},
\]
where $q^d=\prod_a q_a^{d_a},d_a=\left<d,e_a\right>$.
\end{Def}

Let
\[\begin{split}
&\partial_a=q_a\frac{\partial}{\partial q_a}, \mathcal{D}_i=\sum_a m_{ia}\cdot z\partial_a,\\
&\mathcal{P}_d=q^d \prod_{i,\nu:-\left<D_i,d\right>>\nu\geq 0} (\mathcal{D}_i-\nu z)- \prod_{i,\nu:\left<D_i,d\right>>\nu\geq 0} (\mathcal{D}_i-\nu z).
\end{split}
\]
By direct calculations (Lemma 4.6 in \cite{Iritani}) we have:

\begin{lemma}
It holds that $\mathcal{P}_d(I(q,z))=0,\;\forall d\in\mathbb{N}$.
\end{lemma}

Let $\mathrm{Eff}_\mathcal{X}\subseteq H_2(X,\mathbb{Z})$ be the semigroup generated by effective stable maps.
For an arbitrary $\tau\in H_{CR}^\ast(\mathcal{X})$, let $\tau_{0,2}$ be the component of $\tau$ consisting of terms of degree 2, and $\tau'=\tau-\tau_{0,2}$.
Choose an arbitrary basis $\{H_a\}$ of $H_{CR}^\ast(\mathcal{X})$ as a $\mathbb{C}$-module, and $\{H^a\}$ its dual basis. Then $H^a$ could be regarded as in $H_{CR}^\ast(\mathcal{X})$ by applying the Poincare duality.

\begin{Def}
The $J$-function
\[J(\tau,z)=e^{\tau_{0,2}/z} \big(\mathbf{1}+\sum_{\stackrel{(d,\ell)\neq (0,0)}{d\in{\mathrm{Eff}_\mathcal{X}}}}\sum_a \dfrac1{\ell!}\left<\mathbf{1},\tau',\cdots,\tau',\frac{H_a}{z-\psi}\right>^X_{0,\ell+2,d} \cdot e^{\tau_{0,2},d}H^a \big).\]
\end{Def}

Let $I(q,z)=1+\frac{\tau(q)}z+o(z^{-1})$ be the expansion of $I$ with respect to $z$ at $z=\infty$. Here $\tau$ is called the mirror map.

\begin{thms}[Mirror Theorem]It holds that $I(q,z)=J(\tau(q),z)$.
\end{thms}

Mirror theorem for weighted projective space was proved in \cite{Coates1}. General case was proved in \cite{Coates2}.

In our case, by taking $e_a=(0,\cdots,0,1,0,\cdots,0)$, where $1$ is at the $a$-th entry,
$\eta=\sum_{j=-n+1}^{m-1}e_j$ and
\[\begin{array}{lll}
D_{-m}&=\sum_{j=0}^{n-1}(n-j)e_{-j},&\\
D_{-m+\ell}&=ne_\ell, & \ell=1,\cdots,m-1;\\
D_{n-\ell}&=me_{-\ell},& \ell=1,\cdots,n-1;\\
D_n&=\sum_{j=0}^{m-1}(m-j)e_j,&\\
\end{array}
\]
we can also obtain $W\mathbb{P}(m,n)$.
$J(\tau,z)=\sum_a  \left<\!\left<1,\frac{H_a}{z-\psi} \right>\!\right>_{0,2} H^a$ follows directly from Definition 2.2.

On the other hand, consider the Frobenius algebra $V=QH^\ast(W\mathbb{P}(m,n))\cong\mathrm{Jac}(W\mathbb{P}(m,n))$.
Let $H_\alpha$ (which coincides with the $H_a$ above in our case) be its basis as a $\mathbb{C}$ vector space. We identify $V$ and its tangent space $T_pV$ for $p$ a semisimple point on $V$.
Define quantum connection $\nabla_\alpha=z\partial_\alpha -H_\alpha\ast$.
Consider differential equations system: (Quantum Differential Equation, QDE)
\[
\nabla_\alpha h=0,\quad \alpha=1,\cdots,\dim V.
\]
Let $S_\sigma=\sum_a \left<\!\left<H_a,\frac{H_\sigma}{z-\psi} \right>\!\right> H^a$. This forms a set of fundamental solutions to the QDE.

Let $(\alpha,\beta)$ be the Poincare paring $\int_{W\mathbb{P}(m,n)} \alpha \cup \beta$. Then
$(J,H_\sigma) = \left<\!\left<1,\frac{H_\sigma}{z-\psi} \right>\!\right>_{0,2}^{W\mathbb{P}(m,n)} = (S_\sigma,1)$
by the duality of $H_a$ and $H^a$.
Inducting on $k$ with the Leibniz rule, we have
$\big( z\frac{\partial}{\partial t_{i_1}} \cdots z\frac{\partial}{\partial t_{i_k}} J, H_\sigma \big) = (H_{i_1}\cdots H_{i_k}S_\sigma,1)$.

By $\mathcal{P}_d I=0, I=J$, we know that for arbitrary $\sigma$,
\[
\begin{split}
0 & = (\mathcal{P}_d J, H_\sigma) \\
& = \big( \big( q^d\displaystyle\prod_{i,\nu:-\left<D_i,d\right>>\nu\geq 0}(\dsum_a m_{ia}\cdot z\partial_a -\nu z) - \displaystyle\prod_{i,\nu:\left<D_i,d\right>>\nu\geq 0}(\dsum_a m_{ia}\cdot z\partial_a -\nu z)
\big)J,H_\sigma\big) \\
& =\big( \big( q^d\displaystyle\prod_{i,\nu:-\left<D_i,d\right>>\nu\geq 0}(\dsum_a m_{ia}\cdot H_a -\nu z) - \displaystyle\prod_{i,\nu:\left<D_i,d\right>>\nu\geq 0}(\dsum_a m_{ia}\cdot H_a -\nu z) \big)S_\sigma,1\big).
\end{split}
\]

This implies for all $d$,
\[q^d\prod_{i,\left<D_i,d\right>\leq 0}(\sum_a m_{ia}\cdot H_a)^{-\left<D_i,d\right>} - \prod_{i,\left<D_i,d\right>\geq 0}(\sum_a m_{ia}\cdot H_a)^{\left<D_i,d\right>} =0.\]

Take $d=e_a$, let $\bar{X}=(\sum_{j=0}^{m-1}(m-j)H_j)^{\frac1n}$, $X^m=(\sum_{j=0}^{n-1}(n-j)H_{-j})^{\frac1m}$, then the above relations give
$H_\ell  = \dfrac1n q_\ell^{\frac1n}\bar{X}^{-(m-\ell)},\;
\bar{X} X  =q_0^{\frac1{mn}},\;
H_{-\ell}  = \dfrac1m q_{-\ell}^{\frac1m}X^{-(n-\ell)}$.

So we obtain
\[\begin{split}
& \mathbb{C}[H_{-n+1},\cdots,H_{m-1}]/\left<H_\ell - \dfrac1n q_\ell^{\frac1n}\bar{X}^{-(m-\ell)},
\bar{X} X -q_0^{\frac1{mn}}, H_{-\ell} - \dfrac1m q_{-\ell}^{\frac1m}X^{-(n-\ell)} \right> \\
\cong & \mathbb{C}[X,\bar{X}]/\left<X\bar{X}-q_0^{\frac1{mn}},mX^m+ \dsum_{\ell=1}^{m-1}(m-\ell)q_\ell^{\frac1n}q_0^{-\frac{m-\ell}{mn}}X^{(m-\ell)} -n\bar{X}^n-\dsum_{\ell=1}^{n-1}(n-\ell)q_{-\ell}^{\frac1m}q_0^{\frac{n-\ell}{mn}}\bar{X}^{(n-\ell)}\right>.
\end{split}
\]
By dimension argument, this is isomorphic to $QH^\ast(W\mathbb{P}(m,n))$.

\subsection{Equivariant quantum cohomology}
\

In our case, we may recover the equivariant quantum cohomology ring from the equivariant cohomology ring and the quantum cohomology ring.
As a $\mathbb{C}[\bold{w}]$-module, $QH^{\ast}_T(W\mathbb{P}(m,n),\mathbb{C}) \cong H^{\ast}_T(W\mathbb{P}(m,n),\mathbb{C})$. The product structure is given by $(\gamma_1 \ast \gamma_2, \gamma_3) = \left< \!\left<\gamma_1,\gamma_2,\gamma_3\right>\!\right>^{W\mathbb{P}(m,n),T}_{0,3}$.
With a slight abuse of notations, let $X,\bar{X}$ be the image of $X$ and $\bar{X}$ in $QH_T^\ast(W\mathbb{P}(m,n))$ in the $\mathbb{C}$-vector space isomorphism $QH^\ast(W\mathbb{P}(m,n))\stackrel{\sim}{\to}QH_T^\ast(W\mathbb{P}(m,n))$ (after regarding equivariant parameters as complex numbers).

\begin{prop}
It holds that $QH_T^\ast(W\mathbb{P}(m,n))\cong \mathbb{C}[X,\bar{X},p]\{q\}/I(\cong \mathbb{C}[X,\bar{X}]/I)$,

\noindent where
$I=\Bigg< X\bar{X}-q_0^{\frac1{mn}}, mX^m-n\bar{X}^n-\big(- p+\dsum_{\ell=1}^{n-1}(n-\ell)q_{-\ell}^{\frac1m}q_0^{\frac{n-\ell}{mn}}\bar{X}  -\dsum_{\ell=1}^{m-1}(m-\ell)q_\ell^{\frac1n}q_0^{-\frac{m-\ell}{mn}}X^{(m-\ell)}\big)\Bigg>$.
\end{prop}

\begin{proof}
We prove by degree arguments.

First we calculate $X\ast \bar{X}$, where $\ast$ stands for the multiplication in $QH_T^\ast(W\mathbb{P}(m,n))$. We have $X\ast\bar{X}=\sum_{\alpha,\beta} \left<\!\left< X,\bar{X},\gamma_\alpha\right>\!\right>_{0,3,\beta}^{W\mathbb{P}(m,n),T} \gamma^\alpha q^\beta$, where $\gamma_\alpha$ is the basis of $H_{CR}^\ast (W\mathbb{P}(m,n))$, and $\gamma^\alpha$ is its dual basis in $H_T^\ast(W\mathbb{P}(m,n))$. By the compactness of $W\mathbb{P}(m,n)$, we know that $\gamma_\alpha$ can be chosen in $H_{CR}^\ast(W\mathbb{P}(m,n))$. Hence $X\ast \bar{X}$ does not contain equivariant parameters. So by taking non-equivariant limit, it holds that $ X\ast \bar{X}=X\bar{X}=\bold{q_0}^{\frac1{mn}}$.

Next we calculate  $mX^{\ast m}-n\bar{X}^{\ast n}$, where $\alpha^{\ast k}$ means the $k$-th power of $\alpha$ in $QH_T^\ast(W\mathbb{P}(m,n))$ for $k\in\mathbb{Z}_{\geq 0}$ and $\alpha\in QH_T^\ast(W\mathbb{P}(m,n))$. Since $mX^{\ast m}-n\bar{X}^{\ast n}\in QH_T^\ast(W\mathbb{P}(m,n))$, we can write it as
\[mX^{\ast m}-n\bar{X}^{\ast n} = \varphi\big(X,\bar{X},\bold{p}=\dsum_{\ell=1}^m \ell w_\ell - \dsum_{\ell=1}^n \ell w_{-\ell}, \bold{q}=(\bold{q_{-n+1}},\cdots,\bold{q_{m-1}})\big),
\]
where $\varphi$ is a polynomial in $X,\bar{X},\bold{p},\bold{q_0}^{\frac1{mn}},\big( \bold{q_\ell q_0}^{-\frac{m-\ell}m}\big)^{\frac1n}, \big( \bold{q_{-\ell} q_0}^{\frac{n-\ell}n}\big)^{\frac1m}$. By taking large radius limit and non-equivariant limit, we have
\[
\begin{split}
&\varphi(X,\bar{X},\bold{p},0)=-\bold{p},\\
&\varphi(X,\bar{X},0,\bold{q})=\dsum_{\ell=1}^{n-1}(n-\ell)\bold{q_{-\ell}}^{\frac1m}\bold{q_0}^{\frac{n-\ell}{mn}}\bar{X}^{\ast(n-\ell)} -
\dsum_{\ell=1}^{m-1}(m-\ell)\bold{q_\ell}^{\frac1n}\bold{q_0}^{-\frac{m-\ell}{mn}}X^{\ast(m-\ell)}.
\end{split}
\]

Since the degree of $\bold{q_0}^{\frac1{mn}},\big( \bold{q_\ell q_0}^{-\frac{m-\ell}m}\big)^{\frac1n}, \big( \bold{q_{-\ell} q_0}^{\frac{n-\ell}n}\big)^{\frac1m}$ are all positive, and $\deg \bold{p}=\deg \varphi=2$, there are no cross terms of $\bold{p}$ and $\bold{q}$ in $\varphi$. Hence we obtain
\[\varphi(X,\bar{X},\bold{p},\bold{q})= -\bold{p}+\dsum_{\ell=1}^{n-1}(n-\ell)\bold{q_{-\ell}}^{\frac1m}\bold{q_0}^{\frac{n-\ell}{mn}}\bar{X}^{\ast(n-\ell)} -\dsum_{\ell=1}^{m-1}(m-\ell)\bold{q_\ell}^{\frac1n}\bold{q_0}^{-\frac{m-\ell}{mn}}X^{\ast(m-\ell)}.\]

By dimension argument, we drop the superscript $\ast$ and obtain the desired result.
\end{proof}

Let $X^{-1}=q_0^{_{\frac{-1}{mn}}}\bar{X}$.
It is easy to see that $\{X^{m-1},\cdots,X^{-n}\}$ is a flat basis of $QH^\ast_T(W\mathbb{P}(m,n))$.
Inducting on $t$ and applying
\[
X^{m+t}=\dfrac1m X^t\big( nq_0^{\frac1m}X^{-n}+\dsum_{\ell=1}^{n-1}(n-\ell)q_{-\ell}^{\frac1m} X^{-(n-\ell)} -
\dsum_{\ell=1}^{m-1}(m-\ell)q_{\ell}^{\frac1n}q_0^{-\frac{m-\ell}{mn}} X^{(m-\ell)}-p \big),
\]
we know that $\{X^{m+t},\cdots,X^{-n+t+1} \}$ are equivalent (flat) basis for $QH^\ast_T(W\mathbb{P}(m,n))$.
Hence we have $X^{m+n-1},\cdots,X^0 (=1)$ as a flat basis for $QH^\ast_T(W\mathbb{P}(m,n))$.
So $\{X^0,\cdots,X^{m+n-1}\}$ is equivalent flat basis for $QH^\ast_T(W\mathbb{P}(m,n))$, namely, $\{X^0,\cdots,X^{m+n-1}\}$ is a flat basis for $QH^\ast_T(W\mathbb{P}(m,n))$.

\begin{prop}
$QH^\ast_T(W\mathbb{P}(m,n))\cong\mathrm{Jac}(W_T)$ as Frobenius algebras.
\end{prop}

\begin{proof}
Note that this could be proved by applying equivariant mirror theorem \cite{Coates2} or \cite{Fangnew}. However we include a proof which recovers the pairing from its non-equivariant limit as usual.

We identify $\tilde{q}_\ell$ with $q_{m-\ell}^{\frac1{n}}q_0^{-\frac{\ell}{mn}}$ (for $\ell=1,\cdots,m-1$), $\tilde{q}{-_\ell}$ with $q_{-n+\ell}^{\frac1{m}}$ (for $\ell=1,\cdots,n$), $p$ with $\tilde{p}$, and $Y$ with $X$.

Taking non-equivariant limit and applying mirror theorem, we have the isomorphism of non-equivariant limits $QH^\ast(W\mathbb{P}(m,n))\cong\mathrm{Jac}(W)$ as Frobenius algebras. Here the non-equivariant limit is given by $w_i \rightarrow 0$ and $\tilde{w}_i \rightarrow 0$. We now prove that $w_i$ and $\tilde{w_i}$ affect neither the residue pairing nor the Poincare pairing with $1$, which directly leads to the result.

Let $(\cdot,\cdot)$ denote the non-equivariant pairing, $(\cdot,\cdot)_T$ denote the equivariant pairing. For an representative element $g(t) \in \oplus_{j=-n+1}^{m}\mathbb{C}t^j$, we have the following results:

\noindent On the A-side,
\[(g(X),1)=\dint_{\wedge\mathcal{X}}g(X)=\dint_{\wedge\mathcal{X}^T}g(X)=(g(X),1)_T;\]
On the B-side, let $f(Y)=\partial W_T/\partial y$. By the residue formula,
\[(g(Y),1)=-Res_{\infty} \dfrac{g(Y)}{f(Y)}\dfrac{dY}{Y}=\dfrac{1}{2\pi i}\lim _{R \rightarrow \infty} \dint_0^{2\pi}\dfrac{g(Re^{i\theta})}{f(Re^{i\theta})}d\theta.\]
Hence we have
\[(g(Y),1)-(g(Y),1)_T=\dfrac{1}{2\pi i}\lim _{R \rightarrow \infty} \dint_0^{2\pi}\dfrac{g(Re^{i\theta})\cdot pR^ne^{in\theta}}{f(Re^{i\theta})(f(Re^{i\theta})+pR^ne^{in\theta})}d\theta=0.\]
And as a result $(g_1(Y),g_2(Y))_T=(g_1(X),g_2(X))_T$.
\end{proof}

\subsection{Basis}
\

Take $z_i(i=0,\cdots,m+n-1)$ to be the roots of $mz^m +  \sum_{\ell=1}^{m-1}\ell \tilde{q}_\ell z^\ell - \sum_{\ell=1}^{n} \ell \tilde{q}_{-\ell} z^{-\ell} +\tilde{p}=0$ with respect to $z$.
Assume that $z_i$ are distinct.
Let $\phi_i=\prod_{j\neq i}\frac{X-z_j}{z_i-z_j}$. Then $\phi_i$ is a canonical basis, i.e., $\phi_i\cdot\phi_j=\delta_{ij}\phi_i$.

\begin{lemma}
Let $z_i$ be all roots of $f(z)$. Assume $z_i$'s are distinct. Then $\phi_j=\prod_{i\neq j}\frac{z-z_i}{z_j-z_i}$ is a representative of the canonical basis of $\mathbb{C}[z]/\left<f(z)\right>$.
\end{lemma}

\begin{proof}
Observing that the constructed $\phi_i$ is characterized by $\phi_i(z_j)=\delta_{ij}$, we obtain the lemma directly from the Lagrange interpolation formula.
\end{proof}

By direct calculations, we get $\Delta^\alpha=\frac{1}{(\phi_\alpha,\phi_\alpha)}=\frac{m\prod_{\beta\neq\alpha}(z_\alpha-z_\beta)}{z_\alpha^{n-1}}$.
Now consider several different bases for $QH_T^\ast(W\mathbb{P}(m,n),\mathbb{C})$:
\begin{itemize}
  \item The natural basis $T_i=X^i$ and its dual basis $T^i$ with which $(T^i,T_j)=\delta^i_j$.
  \item The canonical basis $\phi_i$ as defined above and its dual basis $\phi^i=\Delta^i(q)\phi_i$.
  \item The normalized canonical basis $\hat{\phi}_i=\sqrt{\Delta^i(q)}\cdot\phi_i$, and its dual basis $\hat{\phi}^i=\hat{\phi}_i$.
\end{itemize}
Regarding $\phi_i$ as a function of $q$, we often write it as $\phi_i(q)$, (same for $\hat{\phi}_i,\phi^i$ and $\hat{\phi}^i$).
For an arbitrary point $pt\in QH_T^\ast(W\mathbb{P}(m,n),\mathbb{C})$, let $t^i,u^i,\bar{u}^i$ be coordinates such that
\[
pt=\sum t^iT_i=\sum u^i\phi_i(q)=\sum\bar{u}^i \phi_i(0).
\]
Regarding $t^i,u^i,\bar{u}^i$ as functions of $q$, we often write them as $t^i(q),u^i(q),\bar{u}^i(q)$. We often call $t^i(q)$ and $\bar{u}^i(q)$ as flat coordinates, and $u^i(q)$ as canonical coordinates.

\section{Graph Sum Formula and the $R$-matrix}

We first demonstrate graph sum formulas for A-model and B-model. These subsections follow \cite{Fang3}.

\subsection{Graph sum formula}
\

Given a connected graph $\Gamma$, we introduce the following notations:
\begin{itemize}
\item Let $V(\Gamma)$ denote the set of vertices in $\Gamma$.
\item Let $E(\Gamma)$ denote the set of edges in $\Gamma$.
\item Let $H(\Gamma)$ denote the set of half edges in $\Gamma$.
\item Let $L^0(\Gamma)$ denote the set of ordinary leaves in $\Gamma$.
\item Let $L^1(\Gamma)$ denote the set of dilation leaves in $\Gamma$.
\end{itemize}
By a half edge we mean either a leaf or an edge, together with a choice of one of the two vertices that it is attached to.
Note that the order of two vertices attached to an edge does not affect the graph sum formula in this paper.
With the above notations, we introduce the following labels:
\begin{itemize}
  \item Genus $g: V(\Gamma)\to \mathbb{Z}_{\geq 0}$;
  \item Marking $\beta: V(\Gamma)\to \{1,\cdots, m+n \}$;
  \item Height $k: H(\Gamma)\to \mathbb{Z}_{\geq 0}$.
\end{itemize}
Note that the marking on $V(\Gamma)$ induces a marking on $L(\Gamma)=L^0(\Gamma)\cup L^1(\Gamma)$ by $\beta(\ell)=\beta(v)$ where $\ell$ is attached to $v$.
Let $H(v)$ be the set of all half edges attached to $v$. Define the valence of $v\in V(\Gamma)$ as $\mathrm{val}(v)=|H(v)|$.
We say a labelled graph $\vec{\Gamma}=(\Gamma,g,\beta,k)$ is stable if $2g(v)-2+\mathrm{val}(v)> 0,\forall v\in V(\Gamma)$.
For a labelled graph $\vec{\Gamma}$, we define the genus by
$g(\vec{\Gamma})=\sum_{v\in V(\Gamma)} g(v)+ (|E(\Gamma)| - |V(\Gamma)| +1)$.
Define $\mathbf{\Gamma}_{g,N}(W\mathbb{P}(m,n))=\{\vec{\Gamma} \; \mathrm{ stable}: g(\vec{\Gamma})=g, |L^0(\Gamma)|=N \}$, and $\mathbf{\Gamma}(W\mathbb{P}(m,n))=\displaystyle\bigcup$$_{_{g,N}}\mathbf{\Gamma}_{g,N}(W\mathbb{P}(m,n))$.
Define the set of graphs $\tilde{\mathbf{\Gamma}}_{g,N}(W\mathbb{P}(m,n))$ in the same manner as $\mathbf{\Gamma}_{g,N}(W\mathbb{P}(m,n))$, except that the $N$ ordinary leaves are ordered.
We define weights for all leaves, edges and vertices and define the weight of a labeled graph $\vec{\Gamma}\in \mathbf{\Gamma}(W\mathbb{P}(m,n))$ to be the product of weights on all leaves, edges and vertices.
A graph sum formula expresses a quantity as sums of weights over all graphs.

\subsection{Givental's formula and the A-model graph sum}
\

For a semisimple Frobenius algebra $V$, let $\{\phi_\alpha\}$ be its canonical basis. Under the identification of $V$ and $T_pV$, $\phi_\alpha$ corresponds to a tangent vector in $T_pV$.
Let $\{u_\alpha\}$ be Givental's canonical coordinates corresponding to $\phi_\alpha$, i.e., $\phi_\alpha=\frac{\partial}{\partial u^\alpha}$.
Let $U=\mathrm{diag} (u_1,\cdots,u_N)$.
Take $\Psi$ to be the base change of $\hat{\phi}_\alpha$ to $T_\alpha$, i.e., $\hat{\phi}_\alpha=\sum_\beta T_\beta\cdot\Psi_\alpha^\beta$.
By Givental's theorem\cite{Given1}, there exists a unitary $R(z)$ (i.e., $R(z)R^T(-z)=\mathrm{id}$) such that $S=\Psi R(z)e^{\frac{U}z}$ is a fundamental solution of the QDE, with $R(z)=\mathrm{id} +R_1 z+\cdots$ a formal power series in $z$.
Furthermore, $R(z)$ is unique up to a right multiplication of $\exp(a_1 z+a_3z^3+a_5z^5+\cdots)$, where $a_i$ are complex diagonal matrices.

The $\mathcal{S}$ operator is given by $(a,\mathcal{S}(b))=\left< \! \left< a, \frac{b}{z-\psi}\right> \! \right>_{0,2}^{W\mathbb{P}(m,n),T}$. The quantization of the $\mathcal{S}$ operator relates the ancestor potential and the descendent potential via Givental's formula\cite{Given}, i.e.,
\[
D^{W\mathbb{P}(m,n),T} (\ub) = \exp \big(F_1^{W\mathbb{P}(m,n),T} \big) \hat{\mathcal{S}}^{-1}A^{W\mathbb{P}(m,n),T}(\ub,\mathbf{t}).
\]
We now describe graph sum formulas for the ancestor potential $A^{W\mathbb{P}(m,n),T}(\ub,\mathbf{t})$ and the descendent potential with arbitrary primary insertions $F_{g,N}^{W\mathbb{P}(m,n),T}(\ub,\mathbf{t})$.

Let $\mathbf{u}=\mathbf{u^\alpha}T_\alpha$.
We assign weights to leaves, edges, and vertices of a labeled graph $\vec{\Gamma}\in \mathbf{\Gamma}(W\mathbb{P}(m,n))$ as
follows.
\begin{enumerate}
\item Ordinary leaves. To each $\ell\in L^0(\Gamma)$  we assign
\[(\mathcal{L}^{\ub})_k^\beta(\ell) = [z^k]( \sum_{\alpha=1}^{m+n}\frac{\ub^\alpha(z)}{\sqrt{\Delta^\alpha (q)}}
R_\alpha^\beta (-z)).\]
\item Dilaton leaves. To each $\ell\in L^1(\Gamma)$ we assign
\[(\mathcal{L}^1)_k^\beta(\ell) = [z^{k-1}]( - \sum_{\alpha=1}^{m+n}\frac{1}{\sqrt{\Delta^\alpha (q)}}
R_\alpha^\beta (-z)).\]
\item Edges. To an edge connecting vertices marked by $\alpha$ and $\beta$, with heights $k$ and $\ell$ at the corresponding half-edges,
we assign
\[\mathcal{E}_{k,\ell}^{\alpha,\beta}(e)=[z^kw^\ell](\frac1{z+w}(\delta_{\alpha,\beta}-\sum_{\gamma=1}^{m+n}
R_\gamma^\alpha(-z)R_\gamma^\beta(-w))).\]
\item Vertices. To a vertex with genus $g$, marking $\beta$ and half-edges of heights $k_1\cdots k_N$, we assign
    \[\mathcal{V}^\beta_g(v) = (\sqrt{\Delta^\beta(q)})^{2g-2+N} \left<\prod_{j=1}^N\tau_{k_j}\right>_g.\]
\end{enumerate}
Hence the weight of $\vec{\Gamma}\in \mathbf{\Gamma}(W\mathbb{P}(m,n))$ is:
\[
w(\vec{\Gamma}) = \dprod_{v\in V (\Gamma)} \mathcal{V}^{\beta(v)}_{g(v)}(v)
\dprod_{e\in E(\Gamma)} \mathcal{E}_{k(h_1(e)),k(h_2(e))}^{\beta(v_1(e)),\beta(v_2(e))} (e) \cdot \dprod_{\ell\in L^0(\Gamma)} (\mathcal{L}^\ub)_{k(\ell)}^{\beta(\ell)}(\ell) \dprod_{\ell\in L^1(\Gamma)} (\mathcal{L}^1)_{k(\ell)}^{\beta(\ell)}(\ell).
\]
Then it holds
\[\begin{split}
\log(A^{W\mathbb{P}(m,n),T}(\ub,\mathbf{t}))&= \dsum_{\vec{\Gamma}\in \mathbf{\Gamma}(W\mathbb{P}(m,n))} \dfrac{h^{g(\vec{\Gamma})-1}w(\vec{\Gamma})}{\mathrm{Aut}(\vec{\Gamma})}= \dsum_{g\geq 0}h^{g-1} \dsum_{N\geq 0}\dsum_{\vec{\Gamma}\in \mathbf{\Gamma}_{g,N}(W\mathbb{P}(m,n))} \dfrac{w(\vec{\Gamma})}{\mathrm{Aut}(\vec{\Gamma})}.
\end{split}
\]

We define a new weight if we have $N$ ordered variables $(\ub_1,\cdots,\ub_N)$ and $N$ ordered ordinary leaves $\{\ell_1,\cdots,\ell_N\}$.
Let
\[
\begin{split}
&S_{\underline{\hat{\alpha}}}^{\hat{\underline{\gamma}}}(z)=(\hat{\phi_\alpha}, \mathcal{S}(\hat{\phi_\alpha})),\ \ub_j=\sum_{a\geq 0}(u_j)_az^a=\sum_\alpha \bold{u_j^\alpha} T_\alpha,\\
&(\stackrel{\circ}{\mathcal{L}^{\ub_j}})_k^\beta(\ell_j) = [z^k]( \sum_{\alpha,\gamma=1}^{m+n}\frac{\ub_j^\alpha(z)}{\sqrt{\Delta^\alpha (q)}} S_{\underline{\hat{\alpha}}}^{\hat{\underline{\gamma}}}(z)
R(-z)_\gamma^\beta).
\end{split}
\]
Let $\stackrel{\circ}{w}({\vec{\Gamma}})$ be the corresponding weight, then it holds similarly
\[\begin{split}
\dsum_{g\geq 0}\hbar^{g-1} \dsum_{N\geq 0} F_{g,N}^{W\mathbb{P}(m,n),T}(\mathbf{u_1,\cdots,u_N,t})= \dsum_{g\geq 0}\hbar^{g-1} \dsum_{N\geq 0} \dsum_{\vec{\Gamma}\in \mathbf{\Gamma}_{g,N}(W\mathbb{P}(m,n))} \dfrac{\stackrel{\circ}{w}(\vec{\Gamma})}{|\mathrm{Aut}(\vec{\Gamma})|}.
\end{split}
\]

\subsection{Eynard-Orantin recursion and the B-model graph sum}

\

Let $\omega_{g,N}$ be defined recursively by the Eynard-Orantin topological recursion
\[
\omega_{0,1}=0,\quad \omega_{0,2}=B(Y_1,Y_2)=\dfrac{dY_1\otimes dY_2}{(Y_1-Y_2)^2}.
\]
When $2g-2+N>0$, we have
\[\begin{split}
\omega_{g,N}(Y_1,\cdots,Y_N) = & \dsum_{\alpha=1}^{m+n} \mathrm{Res}_{Y\to p^\alpha} \dfrac{-\int_{\xi=Y}^{\hat{Y}} B(Y_n,\xi)}{2(\log (Y)-\log (\hat{Y}))dW_T} \big( \omega_{g-1,N-1}(Y,\hat{Y},Y_1,\cdots,Y_{N-1}) \\
& + \dsum_{g_{{}_1}+g_{{}_2}=g} \dsum_{\stackrel{I\cup J=\{1,\cdots,N-1 \}}{I\cap J\neq \emptyset}} \omega_{g_{{}_1},|I|+1}(Y,Y_I) \cdot\omega_{g_{{}_2},|J|+1}(\hat{Y},Y_J) \big),
\end{split}
\]
where $Y\neq p^\alpha$ and $\hat{Y}\neq Y$ are in a small neighborhood of $p^\alpha$ such that $W_T(\hat{Y}) \neq W_T(Y)$.

By definition, it holds that $x=W_T(e^y)$.
Near any critical point $v^\alpha(=\log{p^{\alpha}})$, we define $\zeta_\alpha,h_k^\alpha$ to satisfy
$x=u^\alpha-\zeta_\alpha^2, \ y=v^\alpha-\sum_{k=1}^\infty h_k^\alpha \zeta_\alpha^k$.
Expand $B(p^\alpha,p^\beta)$ in terms of $\zeta_i$ as
\[B(p^\alpha,p^\beta)=\big( \frac{\delta_{\alpha,\beta}}{(\zeta_\alpha -\zeta_\beta)^2}+ \sum_{k,\ell\geq 0} B_{k,\ell}^{\alpha,\beta}\zeta_\alpha^k\zeta_\beta^\ell\big) d\zeta_\alpha\otimes d\zeta_\beta.\]

Let
\[
\begin{split}
&\check{B}_{k,\ell}^{\alpha,\beta}=\frac{(2k-1)!!(2\ell-1)!!}{2^{k+\ell+1}} B_{k,\ell}^{\alpha,\beta},\quad
\check{h}_k^\alpha =\frac{(2k-1)!!}{2^{k-1}}h_{2k-1}^\alpha,\\
&d\xi_k^\alpha=(2k-1)!!2^{-d}Res_{p'\mapsto p^\alpha}B(p,p')(\sqrt{-1}\zeta_\alpha)^{-2d-1}.
\end{split}
\]

The B-model invariants $\omega_{g,N}$ can be expressed as graph sums. Given a labelled graph $\vec{\Gamma}\in \mathbf{\tilde \Gamma}_{g,N}(W\mathbb{P}(m,n))$ with $L^0(\Gamma)=\{\ell_1,\cdots,\ell_N\}$, we define its weight to be
\[
\begin{split}
\tilde{w}(\vec{\Gamma}) = & (-1)^{g(\vec{\Gamma})-1+N} \dprod_{v\in V (\Gamma)} \big(\dfrac{h_1^\alpha}{\sqrt{2}} \big)^{2-2g-\mathrm{val}(v)} \big<\dprod_{h\in H(v)} \tau_{k(h)}\big>_{g(v)}
\dprod_{e\in E(\Gamma)} \check{B}_{k(e),\ell(e)}^{\alpha(v_1(e)),\alpha(v_2(e))} (e) \\ & \cdot \dprod_{j=1}^N \dfrac1{\sqrt{-2}}d\xi_{k(\ell_j)}^{\alpha(\ell_j)}(Y_j) \dprod_{\ell\in L^1(\Gamma)} \big( -\dfrac1{-\sqrt{-2}} \big)\check{h}_{k(\ell)}^{\alpha(\ell)}.
\end{split}
\]

We cite here the Theorem 3.7 in \cite{Dunin}.
\begin{thms}
For $2g-2+N>0$, it holds
\[
\omega_{g,N}=\dsum_{\vec{\Gamma}\in \mathbf{\Gamma}_{g,N}(W\mathbb{P}(m,n))} \dfrac{\tilde{w}(\vec{\Gamma})}{|\mathrm{Aut}(\vec{\Gamma})|}.
\]
\end{thms}

\subsection{A-model large radius limit}

\

In 3.4 and 3.5, we assume $p$ and $z$ to be negative real numbers.

We denote  $Q_1$ as the chart $W\mathbb{P}(m,n)\setminus\{[0:1]\}$, and $Q_2$ as the chart $W\mathbb{P}(m,n)\setminus\{[1:0]\}$.
By Tseng \cite{Tseng10} (see also Zong \cite{Zong}), we have
\[
\left.\big( R_j^i\big)\right|_{t=0,q=0} = \mathrm{diag} \big( (P_\sigma)_j^i\big), \;\mathrm{on} \; Q_\sigma,  \;\mathrm{ for }\; \sigma=1,2;
\]
\[
(P_\sigma)_j^i =\dfrac1{|G_\sigma|}\dsum_{(h)}\chi_{\alpha_j}(h)\chi_{\alpha_i}(h^{-1})\cdot \exp\left[ \dsum_{t=1}^\infty \dfrac{(-1)^t}{t(t+1)}B_{t+1}(c_\sigma(h)) \big( \dfrac{z}{w_\sigma}\big)^t \right].
\]
Let $\sigma=1$. Then we have the following:
\begin{itemize}
  \item $G=G_1=\mathbb{Z}/m\mathbb{Z}$.
  \item $V_{\alpha_{1+j}}=\mathbb{C}$, with $\mathbb{Z}/m\mathbb{Z}$ action: $\bar{t} \circ z=e^{2\pi i \frac{tj}{m}}\cdot z$. Then $\chi_{\alpha_{1+j}}(\bar{t})=e^{2\pi i \frac{tj}{m}}$.
  \item $T=\{ (\lambda_1,\lambda_2)\}$ acts on $Q_1$ by $(\lambda_1,\lambda_2)\circ z = ( \lambda_2^{\frac{n}m}\lambda_1)z$, i.e., $w_1= \lambda_2^{\frac{n}m}\lambda_1$.
  \item $c_\sigma( e^{2\pi i\frac{t}m})=\frac{t}{m}$, where $0\leq t \leq m$.
\end{itemize}

By \cite{Kacz11} we have
$\log \Gamma(z+s)=(z+s-\frac12) \log z -z +\frac12 \log2\pi +\sum_{t=1}^\infty \frac{(-1)^t \cdot B_{t+1}(s)}{t(t+1)}\frac1{z^t}$.
Let $\lambda=\frac{\lambda_1\lambda_2^{n/m}}{z}$. Then we have
\[
(P_1)_\beta^\alpha = \dfrac{\sqrt{2\pi}e^\lambda}{m\sqrt{\lambda}}\dsum_{h=0}^{m-1} \omega^{-h}\Gamma\big(\lambda+\dfrac{h}m\big) \lambda^{1-\lambda-\frac{h}m}.
\]

\subsection{B-model large radius limit}

\

Next we calculate the B-model $R$-matrix $\tilde{R}_\alpha^\beta$ while $\tilde{q}\to 0$.
Let
\[\tilde{R}_\alpha^\beta(\tilde{q})=\sqrt{-2\pi z}\dint_{\gamma_{{}_\beta}}e^{\frac{W_T(Y)-W_T(p^\beta)}z}\theta_\alpha,
\]
where:
\begin{itemize}
  \item $p^\beta$ are critical points of $W_T$ for $\beta=0,\cdots,m+n-1$;
  \item for a fixed $\beta$, $\gamma_{{}_\beta}=W_T^{-1}(W_T(p^\beta)+[0,+\infty))$;
  \item $\theta_\alpha=\frac{dz_\alpha}{z_\alpha^2}$, with $z_\alpha=c_\alpha\cdot(Y-p^\alpha)$ and $c_\alpha\in \mathbb{C}$ such that $W_T(Y)-W_T(p^\alpha)=\frac12z_\alpha^2+o(z_\alpha^2)$.
\end{itemize}
Noticing that $\tilde{R}_\alpha^\beta(\tilde{q})$ only involves terms of differences, it remains unchanged if we add a constant to $W_T$. More specifically, we replace $W_T$ by $W_T-(\sum_{\ell=1}^{m-1} w_\ell\log \tilde{q}_\ell + \sum_{\ell=1}^n w_{-\ell}\log \tilde{q}_{-\ell})$.
Further let $(\tilde{R}_1)_\alpha^\beta$ be the submatrix of the first $m$ columns and $m$ rows of $\lim_{\tilde{q}\to 0}\tilde{R}_\alpha^\beta(\tilde{q})$. It may be computed from $\int_{\gamma_{{}_\beta}}\exp(\frac{\tilde{W_T}(Y)-\tilde{W_T}(p^\beta)}z) \theta_\alpha$ with $\tilde{W_T}=Y^m+p\log Y$.
In this new set-up, we have the following results:
\begin{itemize}
  \item $p^\beta$s are roots of $0=\frac{\partial \tilde{W_T}}{\partial \log Y}=m( Y^m+\dfrac{p}m)$. This gives $p^\beta=\sqrt[m]{-\frac{p}m}\cdot e^{2\pi i \frac\beta{m}}$.
  \item We claim that $\gamma_{{}_\beta}=(0,+\infty)\cdot p^\beta$.
  In fact, direct calculation shows that
  $\tilde{W_T}(t\cdot p^\beta)-\tilde{W_T}(p^\beta)
  =( \frac{-p}m)\cdot ( t^m-1-\log t^m)\geq 0$.

  \item
  $\frac12c_\alpha^2  = \left. \frac{\tilde{W_T}(Y)-\tilde{W_T}(p^\alpha)}{(Y-p^\alpha)^2}\right|_{Y=p^\alpha}= \frac{-mp}{2(p^\alpha)^2}$.
  Taking $c_\alpha=-\frac{\sqrt{-mp}}{2(p^\alpha)^2}$, we get $\theta_\alpha=\frac{-p^\alpha}{\sqrt{-mp}}\cdot \frac{dY}{(Y-p^\alpha)^2}$.
\end{itemize}

Let $\mu=\frac{p}{mz},\omega=e^{2\pi i\frac{\alpha-\beta}{m}},s=-\mu t^m$.
Integrating by parts,
\[\int_{s=0}^{+\infty} e^{-s} \big(\frac{s}\mu \big)^{\mu} d\frac{1}{\omega- \big(\frac{s}\mu \big)^{\frac1m}}
= -\frac1\omega\sum_{h=0}^{m-1} \Gamma\big(\mu+\frac{h}m \big) \cdot \mu^{1-\mu-\frac{h}m}\cdot \omega^{-h}.\]
This equation differs from the main integral by $-\omega e^\mu \frac{\sqrt{2\pi}}{m\sqrt{\mu}} $.
Hence for $1\leq \alpha,\beta \leq m$,
\[
(\tilde{R}_1)_\alpha^\beta = e^\mu \dfrac{\sqrt{2\pi}}{m\sqrt{\mu}}  \dsum_{h=0}^{m-1} \Gamma\big(\mu+\dfrac{h}m \big) \cdot \mu^{1-\mu-\frac{h}m}\cdot \omega^{-h}.
\]

\begin{prop}
It holds that
$(\tilde{R}_1)_\alpha^\beta = (P_1)_\beta^\alpha, \; \mathrm{for }\; 0\leq \alpha,\beta \leq m-1$,
if we identify $\lambda$ and $\mu$.
\end{prop}

\subsection{The general case}

\

It is obvious that $\tilde{R}_\alpha^\beta(q)|_{q=0}=0=\tilde{R}_\beta^\alpha (q)|_{q=0}$ for $1\leq \alpha\leq m < m+1\leq \beta\leq m+n$.
Since both $\Psi Re^{\frac{U}z}$ and  $\Psi \tilde{R}e^{\frac{U}z}$ are solutions to the QDE on the Frobenius algebra $QH^\ast (W\mathbb{P}(m,n))\cong \mathrm{Jac}(W\mathbb{P}(m,n))$, we have by Givental's theorem$\tilde{R}(\tilde{q})=R(q)\cdot A$, where $A=\exp(a_1 z+a_3z^3+a_5z^5+\cdots)$, with $a_i$'s diagonal matrices, and their diagonal entries are scalar.
Considering the submatrix consisting of the first $m$ columns and the first $m$ rows of $R$ and $\tilde{R}$, by the previous proposition we have $\lim_{q\to 0}P_1(q)=\tilde{R}=\lim_{q\to 0}P_1(q)\big|_{m\times m}$.
Comparing the diagonal elements, we find the submatrix of first $m$ columns and $m$ rows of $A$ as $A|_{m\times m}=I_{m\times m}$.
Similarly, moving to the other chart we have $A|_{n\times n}=I_{n\times n}$. Hence $A=I$. This gives the following proposition.
\begin{prop}It holds that
\[R(q)=\tilde{R}(\tilde{q}).\]
\end{prop}

\section{All genus equivariant mirror symmetry}

\subsection{Calculations of the graph sum formula}
\

First observe that
$\sqrt{\Delta^\alpha} h_1^\alpha=\sqrt{2}$.

Let
\[
\begin{split}
&\xi_{\alpha,0}=\frac1{\sqrt{-1}}\sqrt{\frac2{\Delta^\alpha}} \frac{p^\alpha}{Y-p^\alpha},
\theta=\frac{d}{dW_T},
W_k^\alpha =d((-1)^k\theta^k(\xi_{\alpha,0})),\\
&(\tilde{u}_j)_k^\alpha=[z^k]\sum_\beta S^{\hat{\underline{\alpha}}}_{\hat{\underline{\beta}}}(z)\frac{u_j^\beta(z)}{\sqrt{\Delta^\beta(q)}}.
\end{split}
\]
Note that $d\xi_{\alpha,0}=d\xi_0^\alpha$.

\begin{thm}
By identifying $W_k^\alpha(Y_j)$ and $\sqrt{-2}(\tilde{u}_j)_k^\alpha$, we have
\[
\omega_{g,N}=(-1)^{g-1+N}F_{g,N}^{W\mathbb{P}(m,n),T} (\mathbf{u_1,\cdots,u_N,t}).
\]
\end{thm}
\begin{proof} We prove by direct calculation as follows,

\begin{enumerate}
\item{Vertices}:
This follows from $\sqrt{\frac{\Delta^{\alpha(v)}}{2}}h_1^{\alpha(v)}=1$.

\item {Edges}:
By \cite{Dunin},$R_\beta^\alpha(z)=f_\beta^\alpha\big( -\frac1z\big)$ and the contribution of edges to weight in B-model is

\[
\check{B}_{k,\ell}^{\alpha,\beta}(e)
=[u^{-k}v^{-\ell}]\frac{uv}{u+v}\big(\delta_{\alpha\beta}-\sum_{\gamma=1}^{m+n}f_\gamma^\alpha(u)f_\gamma^\beta(v)\big)
=\mathcal{E}_{k,\ell}^{\alpha,\beta}(e).
\]

\item {Ordinary leaves}:
By $\frac1{z+w} = \frac1{z} \sum_{s\geq 0}\big( -\frac{w}z\big)^s$ we know $-\check{B}_{k-1-i,0}^{\alpha,\beta} = [z^{k-i}] R_\beta^\alpha(-z)$.
From \cite{Fang1} we know $d\xi_k^\alpha = W_k^\alpha -\sum_{i=0}^{k-1}\sum_\beta \check{B}_{k-1-i,0}^{\alpha,\beta}W_i^\beta$.

It holds after identifying $\frac1{\sqrt{-2}}W_k^\alpha(Y_j)$ and $(\tilde{u}_j)_k^\alpha$ that
\[
\begin{split}
\big(\mathcal{L}_d^{\bold{u_j}}\big)_{k(\ell_j)}^{\alpha(\ell_j)}(\ell_j)
& = [z^{k(\ell_j)}] \dsum_{\beta=1}^{m+n}\dsum_{i=0}^{k(\ell_j)} (\tilde{u}_j)_i^\beta \cdot z^i \cdot  R_\beta^{\alpha(\ell_j)} (-z)\\
& = \dsum_{i=0}^{k(\ell_j)}\dsum_{\beta=1}^{m+n}(\tilde{u}_j)_i^\beta \big( [z^{k(\ell_j)-i}]R_\beta^{\alpha(\ell_j)} (-z)\big)\\
& = \dfrac1{\sqrt{-2}}d\xi_{k(\ell_j)}^{\alpha(\ell_j)}(Y_j).
\end{split}
\]

\item {Dilaton leaves}:
By \cite{Fang1} and the relation $R_\beta^\alpha(z)=f_\beta^\alpha \big( \frac{-1}z \big)$, we have
\[\check{h}_{k(\ell)}^\alpha = [u^{1-k(\ell)}] \sum_{\beta=1}^{m+n} \sqrt{-1}h_1^\beta R_\beta^{\alpha(\ell)}\big(\frac{-1}{u}\big)
= [z^{k(\ell)-1}] \sum_{\beta=1}^{m+n} \sqrt{-1}h_1^\beta R_\beta^{\alpha(\ell)}\big(-z\big).\]

By $h_1^\beta=\sqrt{\frac2{\Delta^\alpha}}$, we know that
$\big( \mathcal{L}^1\big)_{k(\ell)}^{\alpha(\ell)}(\ell)=\big( -\frac1{\sqrt{-2}}\big)\check{h}_{k(\ell)}^{\alpha(\ell)}$.
\end{enumerate}
\end{proof}

\subsection{The Laplace Transform}
\

Following Iritani \cite{Iritani} with slight modification, we define as follows \cite{Fang3}.
\begin{Def}[equivariant Chern character] We define equivariant Chern character
\[\widetilde{ch}_z: K_T(W\mathbb{P}(m,n))\rightarrow H_T^\ast(W\mathbb{P}(m,n),\mathbb{Q})\left[\left[ \dfrac{p}{z} \right]\right]\] by the following two properties which uniquely characterize it:
\begin{enumerate}
  \item $\widetilde{ch}_z(\varepsilon_1\oplus\varepsilon_2)=\widetilde{ch}_z(\varepsilon_1)+\widetilde{ch}_z(\varepsilon_2)$.
  \item If $\mathcal{L}$ is a $T$-equivariant line bundle on $W\mathbb{P}(m,n)$, then
  $\widetilde{ch}_z(\mathcal{L})=\exp\big( -\frac{2\pi i(c_1)_T(\mathcal{L})}{z}\big)$.
\end{enumerate}
\end{Def}

\begin{Def}[equivariant $K$-theoretic framing]
For $\forall \varepsilon\in K_T(W\mathbb{P}^1(m,n))$, we define the $K$-theoretic framing of $\varepsilon$ by
$\kappa(\varepsilon)= (-z)^{1-\frac{(c_1)_T(W\mathbb{P}(m,n))}{z}}\Gamma\big( 1-\frac{(c_1)_T(W\mathbb{P}(m,n))}{z}\big) \widetilde{ch}_z(\varepsilon)$,
where $(c_1)_T(W\mathbb{P}(m,n)) =mX^m-nX^{-n}+p$.
\end{Def}

\begin{Def}[equivariant SYZ $T$-dual] Let $\mathcal{L}=\mathcal{O}_{W\mathbb{P}(m,n)}(\ell_1 p_1+\ell_2 p_2)$ be an equivariant ample line bundle on $W\mathbb{P}(m,n)$, where $\ell_1,\ell_2 \in \mathbb{Z}$, such that $\ell_1+\ell_2 >0$. We define the equivariant SYZ $T$-dual $\mathrm{SYZ}(\mathcal{L})$ of $\mathcal{L}$ be the figure below:
\end{Def}
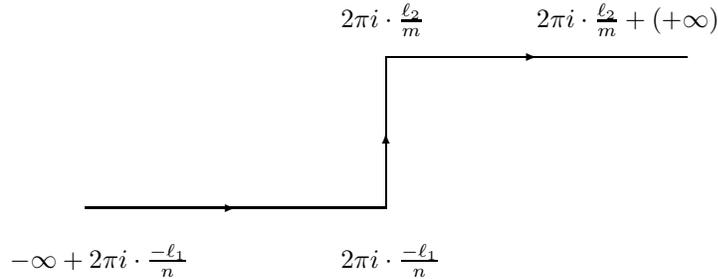
\begin{figure}[h]
\begin{center}
\setlength{\unitlength}{2mm}
\begin{picture}(60,20)

\put(10,6){\vector(1,0){10}}
\put(20,6){\line(1,0){10}}
\put(30,6){\vector(0,1){5}}
\put(30,11){\line(0,1){5}}
\put(30,16){\vector(1,0){10}}
\put(40,16){\line(1,0){10}}

\put(5,2){$-\infty+2\pi i \cdot \frac{-\ell_1}{n}$}
\put(27,2){$2\pi i\cdot\frac{-\ell_1}{n}$}
\put(27,18){$2\pi i\cdot\frac{\ell_2}{m}$}
\put(40,18){$2\pi i \cdot\frac{\ell_2}{m}+(+\infty)$}

\end{picture}
  \label{fig1}
\end{center}
\end{figure}

Extend the definition additively to the equivariant $K$-theory group $K_T(W\mathbb{P}^1(m,n))$.
By \cite{Fang2}:
\begin{thms}
\[
\left<\!\left< \dfrac{\kappa(\mathcal{L})}{z(z-\psi)}\right>\!\right>^{W\mathbb{P}^1(m,n)}_{0,1}=\dint_{\mathrm{SYZ}(\mathcal{L})} e^{\frac{W_T}z} dy.
\]
\end{thms}

\begin{cor}
\

\begin{enumerate}
  \item By string equation: \[
  \dint_{\mathrm{SYZ}(\mathcal{L})} e^{\frac{W_T}z} dy =\left<\!\left< \dfrac{\kappa(\mathcal{L})}{z(z-\psi)}\right>\!\right>^{W\mathbb{P}(m,n),T}_{0,1}
  =\left<\!\left< 1, \dfrac{\kappa(\mathcal{L})}{z-\psi}\right>\!\right>^{W\mathbb{P}(m,n),T}_{0,2};
  \]
  \item Integrating by parts, \[ -\left<\!\left< \dfrac{\kappa(\mathcal{L})}{z-\psi}\right>\!\right>^{W\mathbb{P}(m,n),T}_{0,1} =-z \dint_{\mathrm{SYZ}(\mathcal{L})} e^{\frac{W_T}z} dy = \dint_{\mathrm{SYZ}(\mathcal{L})} e^{\frac{W_T}z} ydx.
  \]
\end{enumerate}
\end{cor}

Define
\[
S_\beta^{\underline{\hat{\alpha}}}(z)=\left<\!\left< \phi_\beta(q),\frac{\hat{\phi}_\alpha(q)}{z-\psi}\right>\!\right>_{0,2}^{W\mathbb{P}(m,n),T},\\
S_{\widehat{\underline{\beta}}}^{\kappa(\mathcal{L)}} (z) =\left<\!\left< \hat{\phi}_\beta(q),\frac{\kappa(\mathcal{L)}}{z-\psi} \right>\!\right>_{0,2}^{W\mathbb{P}(m,n),T}.\]
More generally, we have
\begin{prop}
\[
S_\beta^{\widehat{\underline{\alpha}}}(z)=  -z \dint_{y\in \gamma_\beta(\mathcal{L})} e^{\frac{W_T}z} \dfrac{d\xi_{\alpha,0}}{\sqrt{-2}}.
\]
\[S_{\hat{\underline{\beta}}}^{\kappa(\mathcal{L})}(z)=-z\dint_{y\in{\mathrm{SYZ}(\mathcal{L})}} e^{\frac{W_T}z}\dfrac{d\xi_{\beta,0}}{\sqrt{-2}}.\]
\end{prop}

\begin{proof}

We prove the second equation as an example. The first one may be proved in a similar way.

Let $f(Y)=\frac{\partial W_T}{\partial y}$. Then $\Delta^\alpha=p^\alpha\cdot f'(p^\alpha)$.
By $\hat\phi_\alpha =\sqrt{\Delta^\alpha}\phi_\alpha$ and $\xi_{\alpha,0}=\frac1{\sqrt{-1}} \sqrt{\frac2{\Delta^\alpha}} \frac{p^\alpha}{Y-P^\alpha}$, the desired proposition is equivalent to
\[
\left<\!\left< \phi_\beta,\dfrac{\kappa(\mathcal{L)}}{z-\psi} \right>\!\right>_{0,2}^{W\mathbb{P}(m,n),T}  = z\dint_{y\in{\mathrm{SYZ}(\mathcal{L})}} e^{\frac{W_T}z} d\dfrac{p^\beta}{(Y-P^\beta)\Delta^\beta}.
\]

By Corollary 4.1, we have
\[\left<\!\left< 1,\frac{\kappa(\mathcal{L)}}{z-\psi} \right>\!\right>_{0,2}^{W\mathbb{P}(m,n),T}  = \int_{y\in{\mathrm{SYZ}(\mathcal{L})}} e^{\frac{W_T}z} \frac{dY}Y.\]
Applying $z\dfrac{\partial}{\partial t_i}$ to both sides,  we have
\[\left<\!\left< X^i,\frac{\kappa(\mathcal{L)}}{z-\psi}\right>\!\right>_{0,2}^{W\mathbb{P}(m,n),T} = \int_{y\in\mathrm{SYZ}(\mathcal{L})} Y^i e^{\frac{W_T}z} \frac{dY}Y.\]
Note that for the left hand side derivation, we have opened the double bracket and used the string equation.

This implies
\[\begin{split}
& \left<\!\left< \phi_\beta,\dfrac{\kappa(\mathcal{L)}}{z-\psi} \right>\!\right>_{0,2}^{W\mathbb{P}(m,n),T} \\  = & \dint_{y\in{\mathrm{SYZ}(\mathcal{L})}} e^{\frac{W_T}z} \dfrac{f(Y)Y^n}{Y-p^\beta} \dfrac{dY}Y \cdot \dfrac1{(Y^nf(Y))'|_{p^\beta}}\\
= & -z \dint_{y\in{\mathrm{SYZ}(\mathcal{L})}}\dfrac{1}{(Y-p^\beta)f'(p^\beta)} d e^{\frac{W_T}z} - \dint_{y\in{\mathrm{SYZ}(\mathcal{L})}} e^{\frac{W_T}z} \dfrac{f(Y)(Y^n-(p^\beta)^n)}{(Y-p^\beta)(p^\beta)^nf'(p^\beta)} \dfrac{dY}{Y} \\
= & z \dint_{y\in{\mathrm{SYZ}(\mathcal{L})}} e^{\frac{W_T}z} d\dfrac{1}{(Y-p^\beta)f'(p^\beta)} - \dint_{y\in{\mathrm{SYZ}(\mathcal{L})}} e^{\frac{W_T}z} \dfrac{f(Y)(Y^n-(p^\beta)^n)}{(Y-p^\beta)(p^\beta)^nf'(p^\beta)} \dfrac{dY}{Y}\\
= & z \dint_{y\in \mathrm{SYZ}(\mathcal{L})} e^{\frac{W_T}z} d\dfrac{p^\beta}{(Y-p^\beta)\Delta^\beta}.
\end{split}\]
The last equation holds because
\[\int_{y\in{\mathrm{SYZ}(\mathcal{L})}} e^{\frac{W_T}z} g(Y)f(Y)\frac{dY}Y =
\left<\!\left< g(Y)f(Y),\frac{\kappa(\mathcal{L})}{z-\psi}\right>\!\right>_{0,2}^{W\mathbb{P}(m,n),T}
=\left<\!\left< 0,\frac{\kappa(\mathcal{L})}{z-\psi}\right>\!\right>_{0,2}^{W\mathbb{P}(m,n),T}
=0,\]
where $g(X) =\dfrac{1}{(p^\beta)^nf'(p^\beta)} \dfrac{(Y^n-(p^\beta)^n)}{(Y-p^\beta)}$ (is a polynomial of $Y$).
\end{proof}

Integrating the second equation by parts, we have
\[S_{\hat{\underline{\beta}}}^{\kappa(\mathcal{L})} (z) =-z^{k+1}\dint_{y\in \mathrm{SYZ}(\mathcal{L})}  e^{\frac{W_T}z} \frac{W_k^\beta}{\sqrt{-2}}.\]
Also notice that
\[\begin{split}
&\dsum_{\gamma=1}^{m+n} S_\alpha^{\widehat{\underline{\gamma}}} (z) S_\beta^{\widehat{\underline{\gamma}}}(-z) = (\phi_\alpha(0),\phi_\beta(0))=\Delta^\alpha \delta_{\alpha\beta},\\
&\dsum_{\alpha=1}^{m+n} S_\beta^{\widehat{\underline{\alpha}}} (-z) S_{\widehat{\underline{\alpha}}}^{\kappa(\mathcal{L)}}(z) = (\phi_\beta(0),\mathcal{K(L)}).
\end{split}
\]

\begin{thm} It holds that
\[
\dint_{y_1\in \mathrm{SYZ}(\mathcal{L}_1)}\cdots \dint_{y_\ell\in \mathrm{SYZ}(\mathcal{L}_\ell)} e^{\frac{W_T(y_1)}{z_1}+\cdots \frac{W_T(y_\ell)}{z_\ell}} \omega_{g,\ell} =
(-1)^{g-1} \left<\!\left< \dfrac{\kappa(\mathcal{L}_1)}{z_1-\psi_1}, \cdots,\dfrac{\kappa(\mathcal{L_\ell)}}{z_\ell-\psi_\ell}  \right>\!\right>_{g,\ell}.
\]
\end{thm}

\begin{proof}
By definition,
\[
\tilde{\ub}_j^\alpha(z)=\dsum_{\beta=1}^{m+n}\sqrt{\Delta^\alpha(q)}\left<\!\left< \phi_\alpha(q),\dfrac{\phi_\beta(q)}{z-\psi} \right>\!\right>_{0,2}^{W\mathbb{P}(m,n),T} \ub_j^\beta(z).
\]
Taking the Laplace transform of $\omega_{g,N}$, by Theorem 1 and definition of $\bar{\ub}_i$, we get

\[\begin{split}
& \dint_{y_1\in \mathrm{SYZ}(\mathcal{L}_1)}\cdots \dint_{y_N\in \mathrm{SYZ}(\mathcal{L}_N)} e^{\frac{W_T(y_1)}{z_1}+\cdots \frac{W_T(y_N)}{z_N}} \omega_{g,N} \\
= & \dint_{y_1\in \mathrm{SYZ}(\mathcal{L}_1)}\cdots \dint_{y_N\in \mathrm{SYZ}(\mathcal{L}_N)} e^{\frac{W_T(y_1)}{z_1}+\cdots \frac{W_T(y_N)}{z_N}} (-1)^{g-1-N} \Big( \dsum_{\beta_i,\alpha_i}\left<\!\left< \dprod_{i=1}^N \tau_{\alpha_i}(\phi_{\beta_i}(0)) \right>\!\right>_{g,N} \\
& \cdot \dprod_{i=1}^N (\bar{\ub}_i)_{\alpha_i}^{\beta_i}\Big|_{(\tilde{\ub}_j)_k^{\beta}=\frac1{\sqrt{-2}}W_k^\beta(y_j)} \Big)\\
= & \dint_{y_1\in \mathrm{SYZ}(\mathcal{L}_1)}\cdots \dint_{y_N\in \mathrm{SYZ}(\mathcal{L}_N)} e^{\frac{W_T(y_1)}{z_1}+\cdots \frac{W_T(y_N)}{z_N}}  (-1)^{g-1-N} \Bigg[ \dsum_{\beta_i,\alpha_i}\left<\!\left<\dprod_{i=1}^N \tau_{\alpha_i}(\phi_{\beta_i}(0)) \right>\!\right>_{g,N} \\
& \cdot \dprod_{i=1}^N \big( \dfrac1{\Delta^{\beta_i}} \dsum_{\alpha=1}^{m+n}\dsum_{k\in\mathbb{Z}_{\geq 0}} [z_i^{\alpha_i-k}] S_{\beta_i}^{\widehat{\underline{\alpha}}} (-z_i) \dfrac{W_k^\alpha(y_i)}{\sqrt{-2}} \big) \Bigg]\\
= & (-1)^{g-1+N} \left[\dsum_{\beta_i,\alpha_i}\left<\!\left< \dprod_{i=1}^N\tau_{\alpha_i}(\phi_{\beta_i}(0)) \right>\!\right>_{g,N} \cdot \dprod_{i=1}^N \big( \dfrac1{\Delta^{\beta_i}} \dsum_{\alpha=1}^{m+n}\dsum_{k\in\mathbb{Z}_{\geq 0}} ([z_i^{\alpha_i-k}] S_{\beta_i}^{\widehat{\underline{\alpha}}} (-z_i) )S_{\widehat{\underline{\alpha}}}^{\kappa(\mathcal{L}_1)} (z_i) (-z_i^{-k-1}) \big) \right] \\
= & (-1)^{g-1} \dsum_{\beta_i,\alpha_i}\left<\!\left<\dprod_{i=1}^N \tau_{\alpha_i}(\phi_{\beta_i}(0)) \right>\!\right>_{g,N}  \dprod_{i=1}^N  \dfrac1{\Delta^{\beta_i}} (\phi_{\beta_i}(0),\kappa(\mathcal{L}_i))z_i^{-\alpha_i-1}\\
= & (-1)^{g-1} \left<\!\left< \dfrac{\kappa(\mathcal{L}_1)}{z_1-\psi_1}, \cdots,\dfrac{\kappa(\mathcal{L_N)}}{z_N-\psi_N}  \right>\!\right>_{g,N}.
\end{split}\]
\end{proof}

\

\author{Dun Tang, School of Mathematical Sciences, Peking University, 5 Yiheyuan Road, Beijing 100871, China.
E-mail address: 1400010693@pku.edu.cn}

\end{document}